\documentclass[graybox]{svmult}

\usepackage{mathptmx}       
\usepackage{helvet}         
\usepackage{courier}        
\usepackage{type1cm}        
\usepackage{makeidx}         
\usepackage{graphicx}        
\usepackage{multicol}        
\usepackage[bottom]{footmisc}
\usepackage{amsmath, amsfonts,  amssymb}

\newcommand{\N}{\mathbb{N}}
\newcommand{\Z}{\mathbb{Z}}

\newcommand{\R}{\mathbb{R}}

\renewcommand{\P}{\mathbb{P}}
\renewcommand{\E}{\mathbb{E}}
\newcommand{\A}{\mathcal{A}}

\renewcommand{\S}{\mathcal{S}}

\newcommand{\cL}{\mathcal{L}}

\newcommand{\cX}{\mathcal{X}}

\newcommand{\mc}[1]{{\mathcal #1}}

\newcommand{\bb}[1]{{\mathbb #1}}

\renewcommand{\epsilon}{\varepsilon}

\usepackage{color}

\newtheorem{ass}{\textsc{Assumption}}[section]

\renewcommand\vec[1]{\overrightarrow{#1}}
\newcommand\vecleft[1]{\overleftarrow{#1}}

\makeindex             

\begin{document}

\title*{Second order Boltzmann-Gibbs principle for polynomial functions and applications}

\author{Patr\' icia Gon\c calves, Milton Jara and Marielle Simon}
\institute{Patr\' icia Gon\c{c}alves \at Center for Mathematical Analysis,  Geometry and Dynamical Systems,
Instituto Superior T\'ecnico, Universidade de Lisboa,
Av. Rovisco Pais,1,  1049-001 Lisboa, Portugal.
\at \email{patricia.goncalves@math.tecnico.ulisboa.pt} \and 
Milton Jara \at Instituto de Matem\'atica Pura e Aplicada, Estrada Dona Castorina 110, 22460-320  Rio de Janeiro, Brazil. \at \email{mjara@impa.br} \and 
 Marielle Simon \at \'Equipe MEPHYSTO, Inria Lille -- Nord Europe, 40 avenue du Halley,
59650 Villeneuve d'Ascq, France. \at\email{marielle.simon@inria.fr}}

\maketitle

\abstract{In this paper we give a new proof of the second order Boltzmann-Gibbs principle introduced in \cite{gj2014}. The proof does not impose the knowledge on the  spectral gap inequality for the underlying model and it relies on a proper decomposition of the antisymmetric part of the current of the system in terms of polynomial functions. In addition, we fully derive the convergence of the equilibrium  fluctuations towards 1) a trivial process in case of super-diffusive systems, 2) an Ornstein-Uhlenbeck process or the unique energy solution of the stochastic Burgers equation, as defined in \cite{GubJar,GubPer}, in case of weakly asymmetric diffusive systems. Examples and applications are presented for weakly and partial asymmetric  exclusion processes, weakly  asymmetric speed change exclusion processes and hamiltonian systems with exponential interactions.}

\section{Introduction}
\label{sec:1}

The classical Boltzmann-Gibbs principle, introduced in  \cite{TB}, states that the space-time fluctuations of any local field associated to a conservative model can be written as a linear functional of the conservative field, denoted here by $\mathcal{Y}_t^n$.
 A second-order Boltzmann-Gibbs principle has been introduced in  \cite{gj2014} in order to investigate the first-order correction of this limit, in which case it is given by a quadratic functional of the conservative field $\mathcal{Y}_t^n$. 
 
 In  \cite{gj2014}  the proof of that result was based on a multiscale analysis as done in  \cite{G}, assuming that the underlying particle system is of exclusion type and for which a spectral gap inequality holds. Then, it has been extended to other dynamics, like, for example,  zero-range models in \cite{GJS}. We give here a new proof of that second-order Boltzmann-Gibbs principle, without requiring a  spectral gap inequality. The latter was a crucial ingredient in both \cite{gj2014,GJS}. More precisely, the multiscale analysis was exposed in two main steps: the first one is reminiscent from the well-known one-block estimate, which consists in replacing a local function by its average on a microscopic block;  the second one is reminiscent from the well-known two-blocks estimate and consists in a key iterative bound to replace the aforementioned average on a microscopically big block by an average on a macroscopically small block.

Here we look at specific local fields whose additive functionals can be written as polynomials. We follow step by step the multiscale analysis argument, after decomposing suitably the polynomials, in such a way that there is no need to apply a spectral gap inequality.
In \cite{FGS}, Franco, Gon\c calves and Simon already improved the proof of this second-order Boltzmann Gibbs principle  in order to fit exclusion processes with one slow bond, for which the arguments of \cite{gj2014,GJS} do not apply.

In addition, here we prove the convergence of the fluctuation field $\mc Y_t^n$. Provided that the second order Boltzmann-Gibbs principle is satisfied, we can formulate some simple consequences from it. The first one, is that for super-diffusive systems -- for example,  the asymmetric simple exclusion process -- the density fluctuation field, when properly centered and re-scaled, does not evolve up to a certain time scale.  For diffusive systems -- for example,  weakly asymmetric rates as considered in   \cite{gj2014} -- it is known  that  the sequence of processes $\{\mc Y_t^n\}_{n\in\bb N}$ is tight. From our Botzmann-Gibbs principle we can also prove that any of the limit points of $\{\mc Y_t^n\}_{n\in\bb N}$ is an energy solution of the stochastic Burgers equation as done in  \cite{gj2014}.  In order to characterize the convergence for this type of processes we notice that very recently, Gubinelli and Perkowski \cite{GubPer} obtained the uniqueness of  energy solutions as defined in \cite{GubJar}. From our estimates it is simple to check that the limit points of $\{\mc Y_t^n\}_{n\in\bb N}$ are concentrated on energy solutions of the stochastic Burgers equation  in the sense of \cite{GubPer}, from which the convergence of the sequence $\{\mc Y_t^n\}_{n\in\bb N}$ follows.

To sum up, in this paper we complete the result of \cite{gj2014} and we extend the field of its applications to new interacting particle systems. There are still other models that could be solved by our new approach, like kinetically constrained exclusion processes and zero-range processes, for which the spectral gap inequality is not known. This is a subject for future work.

Here follows an outline of this paper. In Section \ref{sec:framework} we present our main results: the second order Boltzmann-Gibbs principle and the convergence of the sequence of fluctuation fields for systems evolving in different regimes of time. We also give a quick review on the notion of stationary energy solutions of the stochastic Burgers equation and we explain how to obtain the convergence of the fluctuation field to the energy solution of the stochastic Burgers equation, starting from the second-order Boltzmann-Gibbs principle. Section \ref{sec:application} is devoted to applications of this principle to several models. In Sections \ref{sec:BG} and \ref{higher polynomial} we give the complete proof of the main tool, namely the second-order Boltzmann-Gibbs principle, for degree two polynomial functions and higher degree functions, respectively.

\section{Framework and statement of the results}
\label{sec:framework}

In this section we introduce the notation and the main results of this paper. In order to make the presentation as general as we can,  we consider the interacting particle systems evolving in a certain time scale and we detail all the assumptions that we need for our method to work.

\subsection{The microscopic dynamics}

Let $n\in\bb N$ be a scaling parameter and fix $a>0$. We are interested in the evolution of a Markov process $\{\eta^n_{tn^a}(x)\; ;\;x\in\Z,  t\geq 0\}$ in the accelerated time scale $tn^a$, defined through its infinitesimal generator $n^a\cL_n$. This process belongs to the class of conservative one-dimensional interacting particle systems, with state space  $\Omega:=\cX^\Z$.  For instance, if the model is of exclusion type, then $\cX=\{0,1\}$ (so that there is at most one particle per site), whereas for hamiltonian oscillators, $\cX=\R$ or $\R^2$ (the dynamics being on positions and velocities).

We need three assumptions for our method to work. The first one involves the invariant measures, more precisely:

\begin{ass}[Invariant measures] \label{ass1} 
\quad

We assume that the process has a family of invariant measures denoted by $\{\nu_\rho \; ; \; \rho \in I\}$, where $I$ represents the range of values for the parameter. These measures are associated to the conserved quantity: $\sum_{x\in\Z} \eta(x),$ which we call \emph{density}. For any $\rho\in I$ we assume that
\begin{enumerate}
\item[(i)] $\nu_\rho$ is a product measure on $\Omega$;
\item[(ii)] $\nu_\rho$ is invariant by translation, so that $\int_\Omega \eta(x) \nu_\rho(d\eta)=\rho$ for any $x\in\bb Z$;
\item[(iii)] $\nu_\rho$ has finite first moments,  \[\int_\Omega |\eta(x)|^k \nu_\rho(d\eta) < +\infty, \qquad \text{for } k=2,3,4.\]
\end{enumerate}
\end{ass}
\begin{remark} Assumptions $(i)$ and $(ii)$ above imply that $\nu_\rho$ is invariant with respect to the change of variables $\eta \mapsto \eta^{z,z+1}$, for any $z\in\bb Z$, with
  $\eta^{z,z+1} \in \Omega$ given by
  \begin{equation*}
\eta^{z,z+1}(x)=
\begin{cases}\eta(z+1);& x=z,\\
\eta(z);& x=z+1,\\
\eta(x);& x\neq z,z+1.\\
\end{cases}
\end{equation*}
In fact, we do not need to require $\nu_\rho$ to be product, but it should  be at least invariant under the permutation of nearest neighbouring coordinates, and translation invariant. In that case, we would also have to assume that, for any $f \in \bb L^2(\nu_\rho)$ and $v:\mathbb{Z}\to{\mathbb{R}}$, the following bound holds:
\[
\int \bigg(\sum_{x\in\bb Z} v(x) \tau_x f(\eta) \bigg)^2 \nu_\rho(d\eta) \leq C \sum_{x\in\bb Z} v^2(x),
\] where, for $x\in\mathbb{Z}$, we denote by $\tau_x$ the translated operator that acts on a function $h:\Omega\to\bb R$ as $(\tau_x h)(\eta):=h(\tau_x \eta)$, and $\tau_x \eta$ is the configuration obtained from $\eta$ by shifting:  for $y\in\mathbb{Z}$, $(\tau_x \eta)_y=\eta_{x+y}$.  \end{remark}

We denote by $\bar{\eta}(x)=\eta(x)-\rho$ the centered variable and $\chi(\rho)$ the variance:
\[
\chi(\rho):=\int_\Omega \big(\eta(x)-\rho\big)^2 \;\nu_\rho(d\eta).
\]
Let us fix once and for all $\rho \in I$. The generator $\cL_n$ can be decomposed in $\mathbb{L}^2(\nu_\rho)$  into its symmetric and antisymmetric parts, more precisely we write $$\cL_n=\A_n+\S_n,$$ where $$\mathcal{S}_n=(\mathcal{L}_n+\mathcal{L}_n^*)/2\quad \textrm{and} \quad \mathcal{A}_n=(\mathcal{L}_n-\mathcal{L}_n^*)/2,$$
with $\mathcal{L}_n^*$ being the adjoint of $\mathcal{L}_n$ in $\mathbb{L}_n^2(\nu_\rho)$. By the conservation law, for any $x\in\bb Z$, there exists a function $j^n_{x,x+1}$ defined on $\Omega$ such that
\begin{equation*}
\mathcal{L}_n\eta(x)=j^n_{x-1,x}(\eta)-j^n_{x,x+1}(\eta)
\end{equation*}
and $j^n_{x,x+1}$ is called the instantaneous current of the system at the bond $\{x,x+1\}$. To fix notation we denote \begin{equation*}
\mathcal{S}_n\eta(x)=j^{n,s}_{x-1,x}(\eta)-j^{n,s}_{x,x+1}(\eta)\quad\textrm{and}\quad \mathcal{A}_n\eta(x)=j^{n,a}_{x-1,x}(\eta)-j^{n,a}_{x,x+1}(\eta),
\end{equation*}
so that $j^n_{x,x+1}=j^{n,s}_{x,x+1}+j_{x,x+1}^{n,a}$.
We denote by $D_n(f)$ the Dirichlet form associated to the Markov process, which is defined on local functions $f\in\mathbb{L}^2(\nu_\rho)$ as \[D_n(f)=-\int_\Omega f(\eta)\mathcal{L}_nf(\eta)\nu_\rho(d\eta)= -\int_\Omega f(\eta)\mathcal{S}_nf(\eta)\nu_\rho(d\eta).\]
The second assumption that we need is the following:
\begin{ass}[Dirichlet form] \label{ass2} 
\quad

There exists a bounded function $\zeta^n_{0,1}:\Omega\to [\delta,\delta^{-1}]$ with $\delta >0$ such that the Dirichlet form reads as
$D_n(f)=\sum_{z \in \bb Z}  I^n_{z,z+1}(f), $ where
\begin{equation}I^n_{z,z+1}(f):=\int_\Omega\zeta^n_{z,z+1}(\eta)(\nabla_{z,z+1}f(\eta))^2\;\nu_\rho(d\eta),\label{eq:diri}\end{equation}
with $\zeta^n_{z,z+1}(\eta):=\tau_z\zeta^n_{0,1}(\eta)$,  $\nabla_{z,z+1}f(\eta)=f(\eta^{z,z+1})-f(\eta)$.
\end{ass}

Assumption \ref{ass2} may look restrictive but is actually valid for many models of interest. For example, lattice gas dynamics, either symmetric or asymmetric, with positive jumps rates, fall into this category (see Sections  \ref{sec:wasep} and \ref{sec:asep} below).

 The path space of right-continuous and left-limits trajectories with values in $\Omega$ is denoted by $\mathcal{D}(\R_+,\Omega)$. For any initial probability measure $\mu$ on $\Omega$, we denote by $\bb P^n_{\mu}$ the probability measure on  $\mathcal{D}(\R_+,\Omega)$ induced by $\mu$ and the Markov process $\{\eta^n_{tn^a}(x)\; ;\;x\in\Z,\; t\geq 0\}$. If $\mu=\nu_\rho$ we denote $\P^n_{\rho}=\P^n_{\nu_\rho}$ and its expectation by $\E^n_{\rho}$.

Our process of interest is the \emph{density fluctuation field}, defined on functions $H$ in the Schwartz space $\mathcal{S}(\mathbb{R})$, as
\[\mathcal{Y}^n_t(H)=\frac{1}{\sqrt n}\sum_{x\in\mathbb{Z}}H\Big(\frac{x}{n}\Big)\big(\eta^n_{tn^a}(x)-\rho\big).\]
Note that, by Dynkin's formula, for $H\in\mathcal{S}(\mathbb{R})$
\begin{equation*}
M_t^n(H):=\mathcal{Y}^n_t(H)-\mathcal{Y}^n_0(H)-\int_{0}^{t}n^a\mathcal{L}_n(\mathcal{Y}^n_s(H)) ds
\end{equation*}
is a martingale. Let us define
\[
\nabla_nH\Big(\frac{x}{n}\Big):=n\Big[H\Big(\frac{x+1}{n}\Big)-H\Big(\frac{x}{n}\Big)\Big],\ \Delta_nH \Big(\frac{x}{n}\Big):=n\Big[\nabla_nH\Big(\frac{x}{n}\Big)-\nabla_nH\Big(\frac{x-1}{n}\Big)\Big].
\]
 A simple computation shows that the integral part of $M_t^n(H)$ can be written as
\begin{equation}\label{eq:integral}
\mathcal{I}_t^n(H):=\int_{0}^{t}\frac{n^{a-1}}{\sqrt n}\sum_{x\in\mathbb{Z}}\nabla_nH\Big(\frac{x}{n}\Big)\big(j^n_{x,x+1}(\eta^n_{sn^a})-\E^n_{\rho}[j^n_{x,x+1}(\eta)]\big)ds.
\end{equation}
Finally, our last assumption is related to the decomposition of the current.  For a function $\psi:\Omega\to\mathbb{R}$, let us define the centered variable $$\bar{\psi}(\eta):=\psi(\eta)-\E^n_{\rho}[\psi(\eta)].$$
\begin{ass}[Instantaneous current] \label{ass3} 
\quad

There exists a local function $h:=h(n):\Omega\to\mathbb{R}$ and a constant $C:=C(n)$, such that for every $x\in\mathbb{Z}$, $$\bar{j}_{x,x+1}^{n,s}(\eta)=\tau_xh(\eta)-\tau_{x+1}h(\eta)\quad \textrm{and} \quad \bar{j}_{x,x+1}^{n,a}(\eta)=C\bar{\eta}(x)\bar{\eta}(x+1)+\tau_xg(\eta),$$
where $g:=g(n):\Omega\to\mathbb{R}$ is a local function such that for all $ H \in\mathcal {S}(\bb R)$
\begin{equation}\label{term1}
\lim_{n\to\infty}\mathbb{E}^n_{\rho}\Big[\Big(\int_{0}^{t}\frac{n^{a-1}}{\sqrt n}\sum_{x\in\mathbb{Z}}\nabla_nH\Big(\frac{x}{n}\Big)\tau_xg(\eta^n_{sn^a})ds\Big)^2\Big]=0.
\end{equation}
\end{ass}
\begin{remark}
The last condition \eqref{term1} may look strong, but what me mean here is that the error produced by $g$ is small compared to the one induced by the degree two polynomial $\bar\eta(x)\bar\eta(x+1)$, as it is the case for example for polynomials of degree greater or equal than 3 (see Theorem \ref{theo:BG3} below).
\end{remark}
According to the Assumption \ref{ass3} and by a summation by parts, we can rewrite $\mathcal{I}_t^n(H)$ in the following way:
\begin{equation}\label{term2}
\int_{0}^{t}\frac{n^{a-2}}{\sqrt n}\sum_{x\in\mathbb{Z}}\Delta_nH\Big(\frac{x}{n}\Big)\tau_xh(\eta^n_{sn^a})ds+ C\int_{0}^{t}\frac{n^{a-1}}{\sqrt n}\sum_{x\in\mathbb{Z}}\nabla_nH\Big(\frac{x}{n}\Big)\bar{\eta}^n_{sn^a}(x)\bar{\eta}^n_{sn^a}(x+1)ds,
\end{equation}
plus a term which is negligible in $\mathbb{L}^2(\mathbb{P}^n_{\rho})$ and given in \eqref{term1}.
We notice that the first claim of the previous assumption is satisfied by models which are of gradient type. Since for the models of interest $a\leq 2$, to treat the term on the left hand side of \eqref{term2} one can use the classical Boltzmann-Gibbs principle introduced in \cite{TB} and the treatment of the term on the right hand side of \eqref{term2} is the main purpose of this paper. More precisely, we look at the first-order correction for the usual limit projection of space-time fluctuations of the latter specific field. We focus on the additive  functional  of $\bar{\eta}(x)\bar{\eta}(x+1)$  and show how its fluctuations can be written as a linear functional of the conservative field  $\mathcal{Y}^n_t(H)$ plus a quadratic functional of this same field. The crucial point on the proof of this result relies on sharp quantitative bounds on the error that we are able to obtain when we perform the aforementioned replacement.

\subsection{The second-order Boltzmann-Gibbs Principle}
In the following, we simply write $\eta_{tn^a}$ for $\eta^n_{tn^a}$, for the sake of  clarity.
For any $v:\mathbb{Z}\to\bb R$ square summable, we denote:
\begin{equation}\label{vinl2}
 \|v\|_{2,n}^2:=n^{-1}\sum_{x\in\mathbb{Z}}v^2(x)<\infty.\end{equation}
 
\begin{theorem}[Second-order Boltzmann-Gibbs principle]

\label{theo:BG}
 There exists a constant $C=C(\rho)>0$ such that, for any $ L\in\N$ and  $t>0$, and for any function $v\in\ell^2(\mathbb{Z})$:
\begin{multline}
\label{eq:BGexpo}
\mathbb{E}^n_{{\rho}}\Big[\Big(\int_{0}^t \sum_{x\in\mathbb{Z}}v(x)\Big\{\bar{\eta}_{sn^{a}}(x)\bar{\eta}_{sn^{a}}(x+1)-\big(\vec{\eta}_{sn^{a}}^{L}(x)\big)^2+\frac{\chi(\rho)}{L}\Big\}ds\Big)^2\Big]\\
\leq Ct\Big\{\frac{L}{n^{a-1}}+\frac{tn}{L^2}\Big\}\|v\|_{2,n}^2
\end{multline}
where
\begin{equation*}
\vec{\eta}^{L}(x)=\frac{1}{L}\sum_{y=x+1}^{x+L}\bar{\eta}(y).
\end{equation*}
\end{theorem}
Last result can be extended to higher degree polynomials, provided that higher moments are finite: more precisely, if one wants to replace in \eqref{eq:BGexpo} the function $\bar\eta(x)\bar\eta(x+1)$ with a polynomial of degree $d$, then condition $(iii)$ in Assumption \ref{ass1} has to be  replaced by
  \[\int_\Omega |\eta(x)|^k \nu_\rho(d\eta) < +\infty, \qquad \text{for } k=2,...,2d.\]
 This generalization will be the main purpose of Section \ref{higher polynomial} below. Before that, let us present various  of its applications.

\subsection{Consequences of the Boltzmann-Gibbs Principle}

\subsubsection{Super-diffusive systems}

In this section we consider systems which fulfill the  assumptions above  and that evolve super-diffusively so that $a<2$. 
Recall from above that 
\begin{equation}\label{martdecom_1}
\begin{split}
M_t^n(H)=&\mathcal{Y}^n_t(H)-\mathcal{Y}^n_0(H)\\
-&
\int_{0}^{t}\frac{n^{a-2}}{\sqrt n}\sum_{x\in\mathbb{Z}}\Delta_nH\Big(\frac{x}{n}\Big)\tau_xh(\eta^n_{sn^a})ds\\+& C\int_{0}^{t}\frac{n^{a-1}}{\sqrt n}\sum_{x\in\mathbb{Z}}\nabla_nH\Big(\frac{x}{n}\Big)\bar{\eta}^n_{sn^a}(x)\bar{\eta}^n_{sn^a}(x+1)ds,
\end{split}
\end{equation}
plus a term which is negligible in $\mathbb{L}^2(\mathbb{P}^n_{\rho})$ and given in \eqref{term1}.

Since $a<2$ and $h$ is a local function,  a simple computation shows that the first time integral above vanishes in $\mathbb{L}^2(\mathbb{P}^n_{\rho})$, as $n$ goes to infinity. For the second one, we note that by the simple inequality $(x+y)^2\leq 2x^2+2y^2$, the second order Boltzmann-Gibbs principle stated above and  the Cauchy-Schwarz inequality, we can show, as in \cite{G},  that for $a<4/3$ it also vanishes in $\mathbb{L}^2(\mathbb{P}^n_{\rho})$, as $n$ goes to infinity. More details will be given ahead when we apply this result to some concrete  examples.
As a consequence we conclude the triviality of the fluctuations stated in the  next theorem.

\begin{theorem}[Trivial Limit]\label{th:triviality}

For any $a<4/3$, the sequence of processes $\{\mc Y^{n}_t\;;\; t \in [0,T]\}_{n \in \bb N}$ converges in distribution with respect to the Skorokhod topology of $\mc D([0,T]\;;\;\mc S'(\bb R))$, as $n\to\infty$, to the process $\mc Y_t$ given on $H\in\mc S(\bb R)$  by $\mc Y_t(H)=\mc Y_0(H)$.
\end{theorem}

\subsubsection{Diffusive systems}
In this section we consider systems which fulfill the  assumptions above  and that evolve diffusively so that $a=2$. 
Recall \eqref{martdecom_1} and note that if we add a weak asymmetry to the system given by $n^{-\gamma}$, for $\gamma\in{(1/2,1]}$ then, the last integral in the martingale decomposition reads as \begin{equation*}
 \int_{0}^{t}\frac{n^{a-1-\gamma}}{\sqrt n}\sum_{x\in\mathbb{Z}}\nabla_nH\Big(\frac{x}{n}\Big)\bar{\eta}^n_{sn^2}(x)\bar{\eta}^n_{sn^a}(x+1)ds. 
\end{equation*}
In this case, as a consequence of the second order Boltzmann-Gibbs principle stated above, one can show a crossover on the fluctuations which depends on the strength of the asymmetry.

\begin{theorem}[Crossover fluctuations]\label{th:crossover}

The sequence of processes $\{\mc Y^{n}_t\;;\; t \in [0,T]\}_{n \in \bb N}$ converges in distribution with respect to the Skorokhod topology of $\mc D([0,T]\;;\;\mc S'(\bb R))$, as $n\to\infty$, to 

\begin{itemize}
\item an Ornstein-Uhlenbeck process as in \eqref{eq:ou}, if $\gamma>1/2$,
\item an energy solution of the stochastic Burgers equation as in \eqref{eq:SBE_new}, if $\gamma=1/2$.
\end {itemize}
\end{theorem}

Before proceeding any further, we explain in detail what do we mean by energy solution of the stochastic Burgers equation,  and we state the uniqueness result of \cite{GubPer,DieGubPer} regarding these solutions.

\subsubsection{Energy solutions of the stochastic Burgers equation} \label{ssec:uniqueness}
Let us describe the concept of {\em energy solutions} of the stochastic Burgers equation.
Fix $T >0$. Let $\mc S'(\bb R)$ be the Schwartz space of distributions and $\mc C([0,T], \mc S'(\bb R))$ the space of continuous paths in $\mc S'(\bb R)$. We say that a process $\{\mc A_t\; ;\; t \in [0,T]\}$ with trajectories in $\mc C([0,T], \mc S'(\bb R))$  has {\em zero quadratic variation} if the real-valued process $\{\mc A_t(H)\; ;\;  t \in [0,T]\}$ has zero quadratic variation for any test function $H \in \mc S(\bb R)$. Let $\nu, \sigma >0$ and ${\mc W}$ be a space-time standard white noise. Let us denote by  $ \| H\|^2_{\bb L^2(\bb R)}$   the $\bb L^2$-norm of $H$, that is: \[ \|H\|^2_{\bb L^2(\bb R)}:=\int_{\bb R}H^2(x)\, dx.\]

\begin{definition}
We say that a pair of stochastic processes $\{ (\mc Y_t, \mc A_t) \;;\; t \in [0,T]\}$ with trajectories in $\mc C([0,T], \mc S'(\bb R))$ is {\em controlled} by the Ornstein-Uhlenbeck process
\begin{equation}\label{eq:ou}
\partial_t \mc Y_t = \nu \Delta \mc Y_t + \sqrt{2 \nu \sigma^2}\; \nabla {\mc W}_t
\end{equation}
if:
\begin{enumerate}
\item[(i)] for each fixed time $t \in [0,T]$, the $\mc S'(\bb R)$-valued random variable $\mc Y_t$ is a white noise of variance $\sigma^2$,
\item[(ii)] $\mc A_0 \equiv 0$ and the process $\{\mc A_t\; ;\; t \in [0,T]\}$ has zero quadratic variation,
\item[(iii)] for each  $H\in \mc S(\bb R)$, the process
\[M_t(H): = \mc Y_t(H) - \mc Y_0(H) - \int_0^t \mc Y_s(\nu \Delta H) ds- \mc A_t(H)
\]
is a Brownian motion of variance $2 \nu \sigma^2 \|\nabla H\|^2_{\bb L^2(\bb R)}$ with respect to the natural filtration of $(\mc Y_t, \mc A_t)$,
\item[(iv)] the reversed processes $\{(\mc Y_{T-t},\mc A_{T-t}-\mc A_T)\; ; \; t \in [0,T]\}$ also satisfy (iii).
\end{enumerate}\end{definition}

If $\mc A_t \equiv 0$, then $\mc Y_t$ is the unique martingale solution of the Ornstein-Uhlenbeck equation \eqref{eq:ou}. 
The interest of the notion of controlled processes, is that it allows to define some non-trivial functions of the process $\{\mc Y_t\;;\; t \in [0,T]\}$. Let $\{\iota_\epsilon\; ;\;  \epsilon \in (0,1)\}$ be an approximation of the identity and $H\in\mc S(\bb R)$. Then we define the process $\{\mc B_t^\epsilon\; ;\; t \in [0,T]\}$ as
\[
\mc B_t^\epsilon (H): = \int_0^t \int_{\bb R} \big(\mc Y_s \ast \iota_\epsilon (x)\big)^2 \; \nabla H(x) \;dx ds,\]
where $*$ denotes the convolution operator.
The following proposition has been proved in \cite{GubJar,GubPer,DieGubPer}:

\begin{proposition}
Let $\{(\mc Y_t; \mc A_t)\; ;\; t \in [0,T]\}$ be controlled by the Ornstein-Uhlenbeck process given in  \eqref{eq:ou}. Then the limit
\begin{equation}
\label{approxsq}
\mc B_t(f) = \lim_{\epsilon \to 0} \mc B_t^\epsilon(f) 
\end{equation}
exists in $\bb L^2$ and it does not depend on the choice of the approximation of the identity $\{\iota_\epsilon\; ;\; \epsilon \in (0,1)\}$. Moreover, the distribution-valued process $\{\mc B_t\; ;\; t \in [0,T]\}$ defined in this way has zero quadratic variation.
\end{proposition}

This proposition gives a possible way to define the {\em square} of the distribution-valued process $\mc Y_s$. This definition can be used to pose the Cauchy problem for the stochastic Burgers equation.

\begin{definition} Let $\lambda \in \bb R$. 
We say that a stochastic process $\{\mc Y_t\; ;\; t \in [0,T]\}$ is a {\em stationary controlled solution} of the stochastic Burgers equation
\begin{equation}\label{eq:SBE_new}
\partial_t \mc Y_t = \nu \Delta \mc Y_t + \lambda \nabla \mc Y_t^2 + \sqrt{2 \nu \sigma^2}\; \nabla \mc W_t
\end{equation}

if:
\begin{enumerate}
\item[(i)] there exists a process $\{\mc A_t\; ;\; t \in [0,T]\}$ of zero quadratic variation such that $\{(\mc Y_t, \mc A_t)\; ;\;  t \in [0,T]\}$ is controlled by the Ornstein-Uhlenbeck equation \eqref{eq:ou},
\item[(ii)]
$\mc A_t = -\lambda \mc B_t$ for any $t \in [0,T]$.
\end{enumerate}
\end{definition}

The importance of this definition comes from the fact that it gives uniqueness in law for the process $\mc Y_t$ (we refer to \cite{DieGubPer, GubPer} for a proof):

\begin{proposition}
Any two stationary controlled solutions $\{\mc Y_t\; ; \;t \geq 0\}$ and $\{\mc Y_t'\; ;\; t \in [0,T]\}$ of the stochastic Burgers equation have the same law.
\end{proposition}

In the context of interacting particle systems, another notion of solution is more suitable. We say that a process $\{ \mc Y_t\; ;\; t \in [0,T]\}$ with trajectories in $\mc C([0,T], \mc S'(\bb R))$ is {\em stationary} if the $\mc S'(\bb R)$-valued random variable $\mc Y_t$ is a white noise of variance $\sigma^2$ for any $t \in [0,T]$. 

Recall that  $\{\iota_\epsilon\; ;\; \epsilon \in (0,1)\}$ is an approximation of the identity. For each  $H \in \mc S(\bb R)$ and each $\epsilon \in (0,1)$ consider the process $\{\mc B_t^\epsilon(H)\; ;\; t \in [0,T]\}$ as in \eqref{approxsq}. For $s \leq t \in [0,T]$, let  us define $\mc B_{s,t}^\epsilon(H) = \mc B_t^\epsilon(H) - \mc B_s^\epsilon(H)$. 

We say that $\{\mc Y_t\; ;\; t \in [0,T]\}$ satisfies an {\em energy estimate} if there is a finite constant $\kappa >0$ such that
\begin{equation}
\label{energyestimate}
\bb E\Big[ \big( \mc B_{s,t}^\epsilon(f) - \mc B_{s,t}^\delta(f)\big)^2\Big] \leq \kappa \epsilon(t-s) \|\nabla f\|_{\bb L^2(\bb R)}^2
\end{equation}
for any $s \leq t \in [0,T]$, for any $0 <\delta \leq \epsilon <1$ and any $f \in \mc S(\bb R)$. The following proposition has been proved in \cite{gj2014}, but it is also a consequence of the  second order Boltzmann-Gibbs Principle stated above.

\begin{proposition}
\label{energy}
Let $\{\mc Y_t\; ;\; t \in [0,T]\}$ be a process with trajectories in $\mc C([0,T], \mc S'(\bb R))$. Assume that $\{\mc Y_t\; ;\; t \in [0,T]\}$ is stationary and it satisfies an energy estimate. Then the process $\{ \mc B_t\; ;\; t \in [0,T]\}$ given by $\mc B_t(H) = \lim_{\epsilon \to 0} \mc B_t^\epsilon(H)$
is well defined and it satisfies the estimate
\[
\bb E\Big[ \big(\mc B_t(H) - \mc B_s(H)\big)^2 \Big] \leq \tilde \kappa |t-s|^{3/2} \; \|\nabla H\|_{\bb L^2(\bb R)}^2
\]
for some finite constant $\tilde \kappa >0$, for any $s, t \in [0,T]$ and $H\in \mc S(\bb R)$.
\end{proposition}

This proposition gives an alternative way to make sense of the nonlinear term of the stochastic Burgers equation. 

\begin{definition}
We say that a process $\{\mc Y_t\; ;\; t \in [0,T]\}$ with trajectories in $\mc C([0,T], \mc S'(\bb R))$ is a {\em stationary energy solution} of the stochastic Burgers equation if:
\begin{enumerate}
\item[(i)] for each $t \in [0,T]$ the $\mc S'(\bb R)$-valued random variable $\mc Y_t$ is a white noise of variance $\sigma^2$,
\item[(ii)] the process $\{\mc Y_t\; ;\; t \in [0,T]\}$ satisfies an energy estimate,
\item[(iii)] for any $H \in \mc S(\bb R)$ the process
\[
\mc Y_t(H) - \mc Y_0(H) - \int_0^t \mc Y_s(\nu \Delta H) ds + \lambda \mc B_t(H) 
\]
is a Brownian motion of variance $2 \nu \sigma^2 \|\nabla H\|_{L^2(\bb R)}^2$,
\item[(iv)] the reversed process $\{\mc Y_{T-t}\; ;\; t \in [0,T]\}$ also satisfies (iii).
\end{enumerate}
\end{definition}

In \cite{GubJar,GubPer,DieGubPer}, stationary controlled solutions of the stochastic Burgers equation are actually called energy solutions. This is due to the following result:

\begin{proposition}
Let $\{\mc Y_t\; ;\;  t \in [0,T]\}$ be a stationary process satisfying an energy estimate. Then the process $\{\mc B_t\; ;\;  t \in [0,T]\}$ constructed in Proposition \ref{energy} has zero quadratic variation. In particular, the notions of stationary energy solutions and stationary controlled solutions of the stochastic Burgers equation are equivalent.
\end{proposition}

\begin{proof}
First we observe that
\[
\bb E \Big[\sum_{i=1}^n \big(\mc B_{t(i+1)/n}(H) -\mc B_{ti/n}(H)\big)^2\Big] \leq \tilde \kappa \|\nabla H\|_{\bb L^2(\bb R)}^2 \; n^{-1/2}.
\]
This shows that a stationary energy solution is also a stationary controlled solution. But in fact, stationary controlled solutions satisfy stronger energy bounds than \eqref{energyestimate}, see Section 2.3 of \cite{GubPerEBP}.
\end{proof}

\section{Applications to interacting particle systems}
\label{sec:application}

\subsection{Exclusion processes}
\subsubsection{The WASEP and the stochastic Burgers equation} \label{sec:wasep}
For this model we have $\Omega=\{0,1\}^{\Z}$ and the infinitesimal generator is given by
\begin{equation*}
\begin{split}
\mathcal{L}_nf(\eta)=\sum_{x\in\mathbb{Z}}&\Big(\frac{1}{2}+\frac{b}{2n^\gamma}\Big)\eta(x)(1-\eta(x+1))\nabla_{x,x+1}f(\eta)\\+
&\Big(\frac{1}{2}-\frac{b}{2n^\gamma}\Big)\eta(x)(1-\eta(x-1))\nabla_{x,x-1}f(\eta),
\end{split}
\end{equation*}
where $b,\gamma >0$, see for example \cite{gj2014}.
The dynamics conserves the total number of particles and  the invariant measures are given by a family of Bernoulli product measures parametrized by the density $\rho$ which are translation invariant, since for any $x\in\mathbb{Z}$, $\nu_\rho(\eta:\eta(x)=1)=\rho$. Notice that every moment of this measure is finite, so that Assumption \ref{ass1} holds. One can easily check that the Dirichlet form does write on the form \eqref{eq:diri} with
\begin{align*}
\zeta^n_{0,1}(\eta)&=\Big(\frac{1}{2}+\frac{b}{2n^\gamma}\Big)\eta(0)(1-\eta(1))+\Big(\frac{1}{2}-\frac{b}{2n^\gamma}\Big)\eta(1)(1-\eta(0))\\
& = (\eta(0)-\eta(1))^2 + \frac{b}{2n^\gamma}\big[\eta(0)(1-\eta(1))-\eta(1)(1-\eta(0))\big]. \end{align*}
 Moreover one gets \[\bar{j}^n_{x,x+1}(\eta)=\frac{1}{2}\big(\bar{\eta}{(x)}-\bar{\eta}{(x+1)}\big)+\frac{b}{2n^\gamma}\big((\eta(x)-\eta(x+1))^2-2\chi{(\rho)}\big).\]
Therefore, in this case $h(\eta)=\eta(0)/2$, $C(n)=-b/n^\gamma$ and \[g(\eta)=-\frac{b}{2n^\gamma}(1-2\rho)(\bar\eta(x)+\bar\eta(x+1)).\] To simplify the exposition we take $\rho=1/2$ (hence $g(\eta)=0$), nevertheless we notice that by a Galilean transformation, which removes the transport velocity to the system, one could redefine the density fluctuation field and take other values of $\rho$, for more details we refer the reader to, for example, \cite{G}. In that case, the integral part of the martingale \eqref{eq:integral} can be written as
\begin{equation*}
\int_{0}^{t}\frac{n^{a-2}}{2\sqrt n}\sum_{x\in\mathbb{Z}}\Delta_nH\Big(\frac{x}{n}\Big)\bar{\eta}_{sn^a}(x)ds+\int_{0}^{t}\frac{bn^{a-1-\gamma}}{2\sqrt n}\sum_{x\in\mathbb{Z}}\nabla_nH\Big(\frac{x}{n}\Big)\bar{\eta}_{sn^a}(x)\bar{\eta}_{sn^a}{(x+1)}ds.
\end{equation*}
Here the interesting time scale is the diffusive one, namely $a=2$,  so that the previous expression can be written as
\begin{equation*}
\int_{0}^{t}\frac12\mathcal{Y}_s^n(\Delta_nH)ds+\int_{0}^{t}{\frac{bn^{\frac12-\gamma}}{2}}\sum_{x\in\mathbb{Z}}\nabla_nH\Big(\frac{x}{n}\Big)\bar{\eta}_{sn^a}(x)\bar{\eta}_{sn^a}{(x+1)}ds.
\end{equation*}
Now we sketch the proof of Theorem \ref{th:crossover} in this case. 
By Theorem \ref{theo:BG}, with $L=\epsilon n$, together with Young's inequality and a Cauchy-Schwarz inequality, the variance of the term on the right hand side of last expression is bounded above by $\epsilon n^{1-2\gamma}$, which vanishes, as $n\to\infty$, if $\gamma>1/2$. From this it can be shown (see \cite{gj2014})  that for $\gamma>1/2$, the limiting process $\mathcal{Y}_t$ is an Ornstein-Uhlenbeck process. Nevertheless, for $\gamma=1/2$, by Theorem \ref{theo:BG} with $L=\epsilon n$, the term on the right hand side of last expression can be written as
\begin{equation*}
\int_{0}^{t}\frac{1}{n}\sum_{x\in\mathbb{Z}}\nabla_nH\Big(\frac{x}{n}\Big)\Big(\mathcal{Y}_s^n\big(\epsilon^{-1} \textbf{1}_{[0,\epsilon]}\big)\Big)^2 ds,
\end{equation*} plus a term that vanishes in $\mathbb{L}^2(\mathbb{P}^n_\rho)$, as $n\to\infty$ and  $\epsilon\to 0$. From this one can show (see \cite{gj2014} and Subsection \ref{ssec:uniqueness}) that for $\gamma=1/2$, the limiting process $\mathcal{Y}_t$ is the unique energy solution of the stochastic Burgers equation, as stated in Theorem \ref{th:crossover}.

\subsubsection{The ASEP and the time invariance of the density fluctuation field}
\label{sec:asep}
For this model we have $\Omega=\{0,1\}^\Z$ and the infinitesimal generator is given by
\begin{equation*}
\begin{split}
\mathcal{L}_nf(\eta)=\sum_{x\in\mathbb{Z}}&p\eta(x)(1-\eta(x+1))\nabla_{x,x+1}f(\eta)\\+
&(1-p)\eta(x)(1-\eta(x-1))\nabla_{x,x-1}f(\eta),
\end{split}
\end{equation*}
with $p\in (0,1)$, see for example \cite{G}. As above, the dynamics conserves the total number of particles and  the invariant measures are the Bernoulli product measures parametrized by the density $\rho$. Assumption \ref{ass2} on the Dirichlet form holds with
\begin{align*}
\zeta^n_{0,1}(\eta)&=p\eta(0)(1-\eta(1))+(1-p)\eta(1)(1-\eta(0))\\
& = \eta(0)(1-\eta(1)) + (p-1)\big[\eta(0)(1-\eta(1))-\eta(1)(1-\eta(0))\big].
\end{align*}
We also have \[\bar{j}^n_{x,x+1}(\eta)=p{\eta}{(x)}(1-\eta{(x+1)})-(1-p)\eta(x+1)(1-\eta(x))-(2p-1)\chi(\rho).\]  A simple computation shows that
 $\bar{j}^n_{x,x+1}$ can be written as
 \begin{align}
\bar{j}^n_{x,x+1}(\eta)=&-(2p-1)\bar{\eta}(x)\bar{\eta}(x+1)-((1-p)(1-\rho)+p\rho)\big(\bar{\eta}(x+1)-\bar{\eta}(x)\big)\notag \\&+(2p-1)(1-2\rho)\bar{\eta}(x).\label{eq:decompcurrent}
\end{align}
As for the WASEP, we simplify the exposition by assuming $\rho=1/2$, so that the previous expression reads as
 \begin{equation*}
-(2p-1)\bar{\eta}(x)\bar{\eta}(x+1)-\frac{1}{2}\big(\bar{\eta}(x+1)-\bar{\eta}(x)\big),
\end{equation*}
and therefore $h(\eta)=\eta(0)$, $C(n)=-(2p-1)$ and $g(\eta)=0$.
Performing a  summation by parts and by the Cauchy-Schwarz  inequality, the integral part of the martingale \eqref{eq:integral} can be written as
\begin{equation}\label{eq:B}
-(2p-1)\int_{0}^{t}\frac{n^{a-1}}{\sqrt n}\sum_{x\in\mathbb{Z}}\nabla_nH\Big(\frac{x}{n}\Big)\bar{\eta}_{sn^a}(x)\bar{\eta}_{sn^a}{(x+1)}ds,
\end{equation}
plus a term which is negligible in $\mathbb{L}^2(\mathbb{P}^n_{\rho})$ if $a<2$. For this model, the  interesting time scale is a longer time scale than the hyperbolic one, so that we take $a=1+\alpha$, with $\alpha>0$. By Theorem \ref{theo:BG} the variance of the previous term can be estimated doing the following estimates. To fix notation we denote the previous integral by
$\mc B_t^{n}(H)$.
By summing and subtracting $\big(\vec{\eta}_{sn^{a}}^{L}(x)\big)^2-\chi(\rho)/L$ inside the sum above and by the inequality $(x+y)^2\leq 2 x^2+2y^2$ we have that
\begin{align}
&\mathbb{E}^n_{\rho} \big[ \big(\mc B_t^{n}(H)\big)^2\big]\notag  \\
&\leq C\mathbb{E}^n_{\rho}\Big[\Big(\int_{0}^t \frac{n^\alpha}{\sqrt n}\sum_{x\in\mathbb{Z}}\nabla_nH\Big(\frac{x}{n}\Big)\Big[\bar{\eta}_{sn^{a}}(x)\bar{\eta}_{sn^{a}}(x+1)-\big(\vec{\eta}_{sn^{a}}^{L}(x)\big)^2+\frac{\chi(\rho)}{L}\Big]ds\Big)^2\Big]\notag\\
&\quad +C\mathbb{E}^n_{\rho}\Big[\Big(\int_{0}^t \frac{n^\alpha}{\sqrt n}\sum_{x\in\mathbb{Z}}\nabla_nH\Big(\frac{x}{n}\Big)\Big[\big(\vec{\eta}_{sn^{a}}^{L}(x)\big)^2-\frac{\chi(\rho)}{L}\Big]ds\Big)^2\Big].\label{eq:vanishingofB2}
\end{align}
By \eqref{eq:BGexpo} the first expectation is bounded by
\[ C_H \frac{n^{2\alpha}}{n}\Big\{\frac{tL}{n^\alpha}+\frac{t^2n}{L^2}\Big\}.
\]
Now, we treat the remaining expectation. By splitting the sum over intervals of size $L$, by the independence under $\nu_{\rho}$ of $\eta(x)$ and $\eta(y)$ whenever $x\neq y$, and by the Cauchy-Schwarz inequality, \eqref{eq:vanishingofB2} can be bounded from above by
\[\frac{t^2n^{2\alpha}}{n}L\sum_{x\in\bb Z}\Big(\nabla_nH\Big(\frac{x}{n}\Big)\Big)^2 \int_\Omega \Big[\big(\vec{\eta}_{sn^a}^{L}(0)\big)^2-\frac{\chi(\rho)}{L} \Big]^2 \nu_{\rho}(d\eta)\leq C_H \frac{t^2n^{2\alpha}}{L} .\]
Putting together the  previous two estimates, optimizing over $L$, taking $L=n^\theta$ and $\theta=(\alpha+1)/2$, we see that the previous errors vanish as $n\to\infty$, if $\alpha<1/3$.  Therefore we obtain the result of  \cite[Theorem 2.6]{G}, which we recall here:

\begin{theorem}[\cite{G}]\label{theo:BGII}

\quad
Fix $\alpha < 1/3$. For any $H \in \mathcal{S}(\R)$ and $t>0$,
\[\lim_{n\to\infty} \E^n_{\rho}\Big[\Big(\int_0^t\frac{n^\alpha}{\sqrt n} \sum_{x\in\Z} H\Big(\frac{x}{n}\Big) \bar\eta_{sn^{1+\alpha}}(x)\bar\eta_{sn^{1+\alpha}}(x+1) ds\Big)^2\Big] = 0.
\]
\end{theorem}

From this, one can show Theorem \ref{th:triviality} which says  that up to the time scale $n^{4/3}$ the limiting density field $\mathcal{Y}_t$ does not evolve in time. This result is not optimal, it is conjectured that the temporal invariance of the field should go up to the time scale $n^{3/2}$  (see, for example, \cite[Chapter 5]{S}).

\subsubsection{Weakly asymmetric speed change exclusion processes}
\label{sec:porous}

The speed change exclusion processes have been investigated for example in \cite{LV,S}.  Here we look at their weakly asymmetric version so that $\Omega=\{0,1\}^{\Z}$ and the infinitesimal generator is given by
\begin{equation*}
\begin{split}
\mathcal{L}_nf(\eta)=
&\sum_{x\in\mathbb{Z}}c_{x,x+1}(\eta)\Big(\frac{1}{2}+\frac{b}{2n^\gamma}\Big)\eta(x)(1-\eta(x+1))\nabla_{x,x+1}f(\eta)\\+
&\sum_{x\in\mathbb{Z}}c_{x,x-1}(\eta)\Big(\frac{1}{2}-\frac{b}{2n^\gamma}\Big)\eta(x)(1-\eta(x-1))\nabla_{x,x-1}f(\eta),
\end{split}
\end{equation*}
where $b>0$, and the rate functions $c_{x,y}$ satisfy the translation invariance property: there is $c_{0,1}:\Omega \to [\delta,\delta^{-1}]$ such that $c_{x,x+1}=\tau_x c_{0,1}$. For our approach to work, we need to assume:
\begin{enumerate}
\item {\bf (Gradient) } There exists a local function $h:\Omega\to \bb R$ such that, for any $\eta \in \Omega$, \[\eta(0)\big(1-\eta(1)\big)c_{0,1}(\eta)-\eta(1)\big(1-\eta(0)\big)c_{1,0}(\eta)=(h-\tau_1h)(\eta).\]
\item {\bf (Detailed balance) } For any $\eta\in\Omega$, $c_{0,1}(\eta)=c_{0,1}(\eta^{0,1})$.
\item {\bf (Local polynomial) } There exists $\ell,d \in \N$ such that $c_{0,1}$ is a multivariate polynomial of degree $d$ in the variables $(\eta(-\ell),...,\eta(\ell))$.
\end{enumerate}

These three conditions imply that Assumptions \ref{ass1}--\ref{ass3} hold, the invariant measures being (once again) the Bernoulli product measures parametrized by the density $\rho$. However here we only treat one specific example, in order to illustrate how  the generalization of Theorem \ref{theo:BG} to higher degree polynomials can be used.
Let us take:
\[
c_{0,1}(\eta)=\eta(-1)+\eta(2)+1.
\]
 Straightforward computations give
  \begin{multline*}
  \bar{j}^n_{x,x+1}(\eta)=\frac{1}{2}\big(\bar{\eta}{(x)}-\bar{\eta}{(x+1)}\big)\big({\eta}{(x-1)}+{\eta}{(x+2)}+1\big)\\
  +\frac{b}{2n^\gamma}\Big\{(\eta(x)-\eta(x+1))^2\big({\eta}{(x-1)}+{\eta}{(x+2)}+1\big)-2\chi(\rho)\big(2\rho+1\big)\Big\}.\end{multline*}
  A simple computation shows that the symmetric part of the current can be written as the gradient of $h(\eta)$  where
 $$h(\eta)=\frac12\big(\eta(-1)\eta(0)+\eta(0)\eta(1)-\eta(-1)\eta(1)\big)+ \eta(0).$$
  On the other hand a simple but long computation shows that the remaining part of the current, namely  $\bar j_{x,x+1}^{n,a}$ can be written on the form
  \begin{align}
  \bar j^{n,a}_{x,x+1}&(\eta)= -\frac{b}{2n^\gamma} \Big\{2\bar{\eta}(x-1)\bar{\eta}(x)\bar{\eta}(x+1)+2\bar{\eta}(x)\bar{\eta}(x+1)
\bar{\eta}(x+2)\label{intp1}\\
&+(2+4\rho) \bar\eta(x)\bar\eta(x+1)  \label{intp2}\\
& +  (2\rho-1)\big(\bar\eta(x-1)\bar\eta(x)+\bar\eta(x+1)\bar\eta(x+2)\big)\label{intp30}\\
& + (2\rho-1)\big(\bar\eta(x-1)\bar\eta(x+1)+\bar\eta(x)\bar\eta(x+2)\big) \label{intp3}\\
&  + (4\rho^2-1)\big(\bar\eta(x)+\bar\eta(x+1)\big)+2\rho(\rho-1)\big(\bar\eta(x-1)+\bar\eta(x+2)\big)\Big\}.\label{intp4}\end{align}
It is simple to check that the term \eqref{intp4} can be written as a gradient if and only if $\rho$ is solution to $6\rho^2-2\rho-1=0$. We denote by $\rho_0\in{(0,1)}$ the unique solution of that equation. As for the ASEP (see Section \ref{sec:asep}), we now assume $\rho=\rho_0$ to simplify notations, nevertheless we could treat every value of $\rho$ after redefining the density fluctuation field by removing the transport velocity of the system.
 Then, as above, the integral part of the martingale has the term
  \begin{equation}
\int_{0}^{t}\frac{n^{a-2}}{2\sqrt n}\sum_{x\in\mathbb{Z}}\Delta_nH\Big(\frac{x}{n}\Big)\tau_xh({\eta}_{sn^a})ds\label{eq:t1}\end{equation}
 plus terms coming from \eqref{intp2} \eqref{intp30} and \eqref{intp3} of the form
 \begin{equation}
\int_{0}^{t}\frac{bn^{a-1-\gamma}}{\sqrt n}\sum_{x\in\mathbb{Z}}\nabla_nH\Big(\frac{x}{n}\Big)\bar{\eta}_{sn^a}(x)\bar{\eta}_{sn^a}{(x\pm1)}ds \label{eq:t2}
\end{equation}
 and, from \eqref{intp1},
 \begin{equation}
\int_{0}^{t}\frac{bn^{a-1-\gamma}}{\sqrt n}\sum_{x\in\mathbb{Z}}\nabla_nH\Big(\frac{x}{n}\Big)\bar{\eta}_{sn^a}(x)\bar{\eta}_{sn^a}{(x+1)}\bar{\eta}_{sn^a}{(x+2)}ds. \label{eq:t3}
\end{equation}
The interesting time scale is the diffusive one, namely $a=2$. To treat the first term \eqref{eq:t1}, one can use the Boltzmann-Gibbs principle of \cite{TB} and it can be rewritten as
 \begin{equation*}
\int_{0}^{t}\frac{1}{2\sqrt n}\sum_{x\in\mathbb{Z}}\Delta_nH\Big(\frac{x}{n}\Big)h'(\rho)\bar{\eta}_{sn^a}(x)ds
\end{equation*}
where $h'(\rho)=\partial_\rho\mathbb{E}^n_{\rho}[h(\eta)]$.
The terms \eqref{eq:t2} and \eqref{eq:t3} can be treated by the Boltzmann-Gibbs principle proved in Section \ref{sec:BG} and \ref{higher polynomial}, respectively.
 From this it can be shown Theorem \ref{th:crossover} (see \cite{GJS} and Subsection \ref{ssec:uniqueness}) which says that for $\gamma>1/2$, the limiting process $\mathcal{Y}_t$ is an Ornstein-Uhlenbeck process and for $\gamma=1/2$ it is the unique energy solution of the stochastic Burgers equation.

\subsection{Hamiltonian system with exponential interactions}
\label{sec:hamilt}
We consider the same model as introduced in \cite{BG}. In this case  $\Omega=(0,+\infty)^\Z$ and the infinitesimal generator is equal to $\cL_n=\A+\gamma\S$ where $\gamma>0$, and for local differentiable functions $f: \Omega \to \R$ we define
\begin{align*}
(\A f)(\eta) & := \sum_{x\in\Z}\eta_x(\eta_{x+1}-\eta_{x-1})(\partial_{\eta_x}f)(\eta),\\
(\S f)(\eta)& := \sum_{x\in\Z} (f(\eta^{x,x+1})-f(\eta)).
\end{align*}
The invariant measures are given by
\[
\nu_{\beta,\lambda}(d\eta)=\prod_{x\in\Z} \frac{\mathbf{1}_{\{\eta_x>0\}}}{Z_{\beta,\lambda}} \exp\big\{-\beta\eta_x+\lambda\log(\eta_x)\big\}d\eta_x,
\]
 with $Z_{\beta,\lambda}$ being the partition function. These measures are product and translation invariant, and have finite moments. The Dirichlet form writes as
 \[
 D_n(f)=\sum_{z\in\Z}\int_\Omega (\nabla_{z,z+1}f(\eta))^2\;\nu_{\beta,\lambda}(d\eta)
 \] so that, Assumption \ref{ass2} holds with $\zeta_{0,1}(\eta)=1$.  Let $\langle\cdot\rangle_{\beta,\lambda}$ be the average with respect to $\nu_{\beta,\lambda}$. In this case we have
\begin{align*}
\rho&=\rho(\beta,\lambda)=\langle \eta_0\rangle_{\beta,\lambda}=(\lambda+1)\beta^{-1}\\
\chi& =\chi(\beta,\lambda)=\langle\eta_0^2\rangle_{\beta,\lambda}-\langle\eta_0\rangle^2_{\beta,\lambda} = (\lambda+1)\beta^{-2}.
\end{align*}
The microscopic current is given by
 \[
 \bar{j}^n_{x,x+1}(\eta)=-\bar{\eta}(x)\bar{\eta}(x+1)-\gamma \; (\bar{\eta}(x+1)-\bar{\eta}(x))-\rho(\bar\eta(x)+\bar\eta(x+1))-\rho^2, \]
 so that $h(\eta)=\gamma \eta(0)$, $C(n)=-1$ and $g(\eta)=-\rho(\bar\eta(x)+\bar\eta(x+1))-\rho^2$. Once again, for the sake of simplicity, one can assume $\rho=0$, and the general case $\rho \neq 0$ can be solved by redefining the fluctuation field using the Galilean transformation.

Here the interesting time scale is longer that the hyperbolic one, and from this point one can repeat exactly the same arguments which are detailed in Section \ref{sec:asep} for the ASEP: the integral part of the martingale can be written as in \eqref{eq:B} with $2p-1=1$, and we easily get the statement  of  \cite[Theorem 4]{BG}, whose conclusion reads as in Theorem \ref{theo:BGII} and whose proof
 is similar to the one given in Subsection \ref{sec:asep}.

\section{Proof of the Second-order Boltzmann-Gibbs principle}\label{sec:BG}

In this section we present a proof of  the second-order Boltzmann-Gibbs principle stated in Theorem \ref{theo:BG}, which is  the main result of this work. For that purpose we derive several estimates that are needed in what follows. To keep notation simple in the following arguments, we let $C:=C(\rho)$  denote  a constant (that does not depend on $n$ nor on $t$ nor on the sizes of the boxes involved) that may change from line to line.  For $\ell\in\mathbb{N}$ and $x\in\mathbb{Z}$, we introduce two empirical averages on boxes of sixe $\ell$, the first one being to the right of the site $x$, the second one being to the left:
\[
\vec\eta^\ell(x)=\frac{1}{\ell}\sum_{y=x+1}^{x+\ell} \eta(y), \qquad \vecleft\eta^\ell(x):=\frac{1}{\ell}\sum_{y=x-\ell}^{x-1} \eta(y).
\]
For a function $\varphi: \Omega \to \mathbb R$ we denote by $\|\varphi\|_2^2$ its  $\bb L^2(\nu_\rho)$-norm:

$$\|\varphi\|^2_2=\int_\Omega(\varphi(\eta))^2\;\nu_\rho(d\eta).$$

\subsection{Auxiliary estimates} \label{ssec:oneblock}

\begin{proposition} [One-block estimate]
\label{prop:one-block}
Fix $\ell_0\in \mathbb N$ and let $\varphi,\psi: \Omega \to \mathbb R$ be  local
functions which have mean zero with respect to $\nu_\rho$, and such that \begin{enumerate}
\item the support of $\varphi$ does not intersect the set of points 
${\{0,\cdots,\ell_0\}}$ ,
\item the support of $\psi$ does not intersect the set of points $\{-\ell_0, \cdots,-1\}$.
\end{enumerate} There exists a constant $C$, such that for any $t>0$ and any function $v\in\ell^2(\mathbb{Z})$:
\[
\mathbb{E}^n_{\rho}\Big[\Big(\int_{0}^t ds \sum_{x\in\mathbb{Z}}v(x)\varphi(\tau_x\eta_{sn^a})\big(\bar{\eta}_{sn^a}(x+1)-\vec{\eta}_{sn^a}^{\ell_0}(x)\big)\Big)^2\Big] \leq
 C\frac{t\ell_0^2}{n^{a-1}}\|\varphi\|_2^2\; \|v\|_{2,n}^2,
\]
\[
\mathbb{E}^n_{\rho}\Big[\Big(\int_{0}^t ds \sum_{x\in\mathbb{Z}}v(x)\psi(\tau_x\eta_{sn^a})\big(\bar{\eta}_{sn^a}(x)-\vecleft{\eta}_{sn^a}^{\ell_0}(x)\big)\Big)^2\Big] \leq
 C\frac{t\ell_0^2}{n^{a-1}}\|\psi\|_2^2\; \|v\|_{2,n}^2.
\]
\end{proposition}

\begin{proof}
We start by proving the first estimate above.
By  \cite[Lemma 2.4]{KLO}, we bound the previous expectation from above by
\begin{equation*}
 Ct\Big\|\sum_{x\in\mathbb{Z}}v(x)\varphi(\tau_x\eta)\big(\bar{\eta}(x+1)-\vec{\eta}^{\ell_0}(x)\big)\Big\|_{-1}^2,\end{equation*}
where the $H_{-1}$-norm is defined through a variational formula, and in particular the previous expression is equal to
\begin{equation}
Ct\;\sup_{f\in {\bb L}^2(\nu_\rho)}\Big\{2\int \sum_{x\in\mathbb{Z}}v(x)\varphi(\tau_x\eta)\big(\bar{\eta}(x+1)-\vec{\eta}^{\ell_0}(x)\big)f(\eta)\nu_\rho(d\eta)-n^aD_n(f)\Big\}. \label{eq:varia}
\end{equation}
Now, since $\bar{\eta}(x+1)-\vec{\eta}^{\ell_0}(x)$ is written  as the gradient:
\begin{equation*}
\bar{\eta}(x+1)-\vec{\eta}^{\ell_0}(x)=\frac{1}{\ell_0}\sum_{y=x+2}^{x+\ell_0}\sum_{z=x+1}^{y-1}(\bar{\eta}(z)-\bar{\eta}(z+1)),
\end{equation*}
 we can write the integral written in \eqref{eq:varia} as twice its half and in one of the terms we make the exchange $\eta$ to $\eta^{z,z+1}$, for which the measure $\nu_\rho$ is invariant, to get that 
\begin{align}
&2\int\sum_{x\in\mathbb{Z}}v(x)\varphi(\tau_x\eta)\Big\{\frac{1}{\ell_0}\sum_{y=x+2}
^{x+\ell_0}\sum_{z=x+1}^{y-1}(\bar{\eta}(z)-\bar{\eta}(z+1))\Big\}f(\eta)\nu_\rho(d\eta)\notag\\
& =  2\int\sum_{x\in\mathbb{Z}}v(x)\varphi(\tau_x\eta)\Big\{\frac{1}{\ell_0}\sum_{y=x+2}
^{x+\ell_0}\sum_{z=x+1}^{y-1}(\bar{\eta}(z+1)-\bar{\eta}(z))\Big\}f(\eta^{z,z+1})\nu_\rho(d\eta)\notag\\
 & =\int\sum_{x\in\mathbb{Z}}v(x)\varphi(\tau_x\eta)\Big\{\frac{1}{\ell_0}\sum_{y=x+2}
^{x+\ell_0}\sum_{z=x+1}^{y-1}(\bar{\eta}(z)-\bar{\eta}(z+1))\Big\}(f(\eta)-f(\eta^{z,z+1}))\nu_\rho(d\eta).\notag
\end{align}
By Young's inequality, for any sequence $(A_x)_{x\in\mathbb{Z}}$ of positive real numbers, the last integral  is bounded by
\begin{align}&\frac{1}{\ell_0}\sum_{x\in\mathbb{Z}}\sum_{y=x+2}
^{x+\ell_0}\sum_{z=x+1}^{y-1}v(x)\frac{A_x}{2}\int(\varphi(\tau_x\eta))^2\frac{(\bar{\eta}(z)-\bar{\eta}(z+1))^2}{\zeta^n_{z,z+1}(\eta)}\nu_\rho(d\eta) \label{eq:first}\\
+&\frac{1}{\ell_0}\sum_{x\in\mathbb{Z}}\sum_{y=x+2}
^{x+\ell_0}\sum_{z=x+1}^{y-1}\frac{v(x)}{2A_x}\int\zeta^n_{z,z+1}(\eta)\big(f(\eta)-f(\eta^{z,z+1})\big)^2\nu_\rho(d\eta). \label{eq:second}
\end{align}
By taking $2A_x=\ell_0 v(x)/n^a$ and by independence of $\eta(x), \eta(y)$ (for $x\neq y$) with respect to $\nu_\rho$,  the first term \eqref{eq:first} is bounded by
\begin{equation}\label{error_1}
C\frac{1}{n^a}\sum_{x\in\mathbb{Z}}\sum_{y=x+2}
^{x+\ell_0}\sum_{z=x+1}^{y-1}v^2(x)\|\varphi\|_2^2\leq C\frac{\ell_0^2}{n^a}\sum_{x\in\mathbb{Z}}v^2(x)\|\varphi\|_2^2,
\end{equation}
for some positive constant $C(\rho)$. From Lemma \ref{lemma:dir1} (proved at the end of this section), the second term \eqref{eq:second} is bounded by
\begin{equation}\label{error_2}
\frac{n^a}{\ell_0^2}\sum_{x\in\mathbb{Z}}\sum_{y=x+2}
^{x+\ell_0}\sum_{z=x+1}^{y-1}I^n_{z,z+1}(f) \le {n^a}D_n(f).
\end{equation}
Putting together \eqref{error_1}, \eqref{error_2}, the proof of the first estimate ends.
For the second estimate we notice that  the function $\bar{\eta}(x)-\vecleft\eta^\ell(x)$ can also be written as a gradient as above. Then the same argument applies.
\end{proof}

\begin{proposition}[Doubling the box]\label{doub box}
Fix $\ell_k\in\mathbb{N}$, and define $\ell_{k+1}=2\ell_k$. Let $\varphi,\psi: \Omega \to \mathbb R$ be  local
functions which have mean zero with respect to $\nu_\rho$, and such that \begin{enumerate}
\item the support of $\varphi$ does not intersect the set of points 
${\{0,\cdots,\ell_{k+1}\}}$ ,
\item the support of $\psi$ does not intersect the set of points $\{-\ell_{k+1}, \cdots,-1\}$.
\end{enumerate}
There exists a constant $C$, such that for any $t>0$ and any function $v\in\ell^2(\mathbb{Z})$:
\begin{align*}
\mathbb{E}^n_{\rho}\Big[\Big(\int_{0}^t ds \sum_{x\in\mathbb{Z}}v(x)\varphi(\tau_x\eta_{sn^a})\big(\vec{\eta}_{sn^a}^{\ell_k}(x)-\vec{\eta}_{sn^a}^{\ell_{k+1}}(x)\big)\Big)^2\Big]
&\leq
 C\frac{t\ell_k^2}{n^{a-1}}\|\varphi\|_2^2\|v\|_{2,n}^2,\\
\mathbb{E}^n_{\rho}\Big[\Big(\int_{0}^t ds \sum_{x\in\mathbb{Z}}v(x)\psi(\tau_x\eta_{sn^a})\big(\vecleft{\eta}_{sn^a}^{\ell_k}(x)-\vecleft{\eta}_{sn^a}^{\ell_{k+1}}(x)\big)\Big)^2\Big]
&\leq
 C\frac{t\ell_k^2}{n^{a-1}}\|\psi\|_2^2\|v\|_{2,n}^2. 
\end{align*}
\end{proposition}

\begin{proof}
As above, we only write the proof for the first estimate since for the other one it is completely analogous. As in Proposition \ref{prop:one-block}, the important fact is that
\begin{equation*}
\vec{\eta}^{\ell_k}(x)-\vec{\eta}^{\ell_{k+1}}(x)=\frac{1}{2\ell_{k}}\sum_{y=x+1}^{x+\ell_k}(\bar{\eta}(y)-\bar{\eta}(y+\ell_{k})).
\end{equation*}
By \cite[Lemma 2.4]{KLO}, the change of variables $y \mapsto y-x$ and a standard convexity inequality, we can bound  from above the first expectation appearing in the statement of the proposition  by
\begin{equation}\label{eq:init}
Ct\ell_k\sum_{y=1}^{\ell_k}\Big\|\sum_{x\in\mathbb{Z}}v(x)\varphi(\tau_x\eta)\frac{1}{2\ell_{k}}(\bar{\eta}(y+x)-\bar{\eta}(y+x+\ell_{k}))\Big\|_{-1}^2,\end{equation}
which is equal to 
\begin{multline*}
Ct\ell_k\sum_{y=1}^{\ell_k}\sup_{f\in {\bb L}^2(\nu_\rho)}\Big\{2\int \sum_{x\in\mathbb{Z}}v(x)\varphi(\tau_x\eta)\frac{1}{2\ell_{k}}(\bar{\eta}(y+x)-\bar{\eta}(y+x+\ell_{k}))f(\eta)\nu_\rho(d\eta)\\-n^aD_n(f)\Big\}.
\end{multline*}
As above, the function $\bar{\eta}(y+x)-\bar{\eta}(y+x+\ell_k)$ can be written as a gradient:
\begin{equation*}
\bar{\eta}(y+x)-\bar{\eta}(y+x+\ell_k)=\sum_{z=y+x}^{y+x+\ell_k-1}(\bar{\eta}(z)-\bar{\eta}(z+1)),
\end{equation*}
and by writing the previous integral as twice its half  and making, in one of the terms,  the exchange $\eta$ to $\eta^{z,z+1}$, which lets $\varphi(\tau_x\eta)$ invariant for any $z$ that is involved, we get
\begin{equation*}
\begin{split}
&2\int \sum_{x\in\mathbb{Z}}v(x)\varphi(\tau_x\eta)\frac{1}{2\ell_k}(\bar{\eta}(y+x)-\bar{\eta}(y+x+\ell_k))f(\eta)\nu_\rho(d\eta)\\
 & = \sum_{x\in\mathbb{Z}}v(x)\varphi(\tau_x\eta)\frac{1}{2\ell_k}\sum_{z=y+x}^{y+x+\ell_k-1}(\bar{\eta}(z)-\bar{\eta}(z+1))(f(\eta)-f(\eta^{z,z+1}))\nu_\rho(d\eta).
\end{split}
\end{equation*}
By Young's inequality we bound the last expression above by
\begin{equation*}
\begin{split}
& \sum_{x\in\mathbb{Z}}\sum_{z=y+x}^{y+x+\ell_k-1}\frac{v(x)A_x}{4\ell_k}\int \big(\varphi(\tau_x\eta)\big)^2\frac{(\bar{\eta}(z)-\bar{\eta}(z+1))^2}{\zeta^n_{z,z+1}(\eta)}\nu_\rho(d\eta)\\
+& \sum_{x\in\mathbb{Z}}\sum_{z=y+x}^{y+x+\ell_k-1}\frac{v(x)}{4\ell_k A_x}\int \zeta^n_{z,z+1}(\eta)(f(\eta)-f(\eta^{z,z+1}))^2\nu_\rho(d\eta).
\end{split}
\end{equation*}
By taking $4A_x=v(x)/n^a$ and doing similar estimates to the ones of the previous proposition we bound last expression by
\begin{equation*}
\frac{C}{ \; n^a}\|\varphi\|_2^2\sum_{x\in\mathbb{Z}}v^2(x)+\frac{n^a}{\ell_k}\sum_{x\in\mathbb{Z}}\sum_{z=y+x}^{y+x+\ell_k-1}I^n_{z,z+1}(f).
\end{equation*}
Now, invoking Lemma \ref{lemma:dir1},  \eqref{eq:init} is bounded from above by
\begin{equation*}
C\frac{t\ell_k}{\; n^{a-1}}\sum_{y=1}^{\ell_k}\|\varphi\|_2^2\|v\|_{2,n}^2\le \frac{Ct\ell_k^2}{n^{a-1}}\|\varphi\|_2^2\|v\|_{2,n}^2,
\end{equation*}
which proves the result.
\end{proof}
From the previous results  we obtain, by similar arguments, the following result:

\begin{corollary} 
\label{prop:one-block3}
Fix $\ell_0, M\in\bb {N}$,  and let $\varphi: \Omega \to \mathbb R$ be a  local
function which has mean zero with respect to $\nu_\rho$, and whose support does not intersect the set of points 
$\{-2^{M-1}\ell_0,\cdots,-1\}$.

There exists a constant $C$, such that for any $t>0$ and any function $v\in\ell^2(\mathbb{Z})$:
\[
\mathbb{E}^n_{\rho}\Big[\Big(\int_{0}^t ds \sum_{x\in\mathbb{Z}}v(x)\varphi(\tau_x\eta_{sn^a})(x)\big(\vecleft{\eta}^{2^{M-1}\ell_0}_{sn^a}(x)-\vecleft{\eta}_{sn^a}^{\ell_0}(x)\big)\Big)^2\Big] \leq
 C\frac{t\ell_0^2}{\; n^{a-1}}\|\varphi\|_2^2\|v\|_{2,n}^2.
\]
\end{corollary}

\begin{proposition}\label{prop:decomp}
There exists a constant $C(\rho)>0$ such that, for any   $t>0$ and any function $v\in\ell^2(\mathbb{Z})$:
\begin{multline*}
\mathbb{E}^n_{\rho}\Big[\Big(\int_{0}^t ds \sum_{x\in\mathbb{Z}}v(x)\Big\{\bar{\eta}_{sn^a}(x)\vec{\eta}_{sn^a}^{L}(x)-(\vec{\eta}_{sn^a}^{L}(x))^2+\frac{1}{2L}
\big(\bar{\eta}_{sn^a}(x)-\bar{\eta}_{sn^a}(x+1)\big)\Big\}\Big)^2\Big]\\ \leq
 C(\rho)\frac{tL}{n^{a-1}}\|v\|_{2,n}^2.
\end{multline*}
\end{proposition}
\begin{proof}
As above we use \cite[Lemma 2.4]{KLO} and we repeat exactly the same steps, in order to get the variational formula: 
\begin{multline}\label{eq:multline}
\sup_{f\in\bb L^2(\nu_\rho)} \bigg\{2\int \sum_{x\in\mathbb{Z}}v(x)\Big\{\bar{\eta}(x)\vec{\eta}^{L}(x)-(\vec{\eta}^{L}(x))^2+\frac{1}{2L}
\big(\bar{\eta}(x)-\bar{\eta}(x+1)\big)\Big\}f(\eta)\nu_\rho(d\eta)\\ - n^a D_n(f)\bigg\}
\end{multline}
 We notice that
\begin{align}
&2\int \sum_{x\in\mathbb{Z}}v(x)\vec{\eta}^L(x)\big\{\bar{\eta}(x)-\vec{\eta}^L(x)\big\}f(\eta)\nu_\rho(d\eta)\label{eq:firstpart}\\
& = 2\int \sum_{x\in\mathbb{Z}}v(x)\bar{\eta}^L(x)\Big\{\bar{\eta}(x)-\bar{\eta}(x+1)+
\frac{L-1}{L}\big(\bar{\eta}(x+1)-\bar{\eta}(x+2)\big)\notag\\
 & +\cdots+\frac{1}{L}\big(\bar{\eta}(x+L-1)-\bar{\eta}(x+L)\big)\Big\}f(\eta)\nu_\rho(d\eta)\notag\\
&=2\int \sum_{x\in\mathbb{Z}}v(x)\vec{\eta}^L(x)\big\{\bar{\eta}(x)-\bar{\eta}(x+1)\big\}f(\eta)\nu_\rho(d\eta)\notag\\
&+2\int \sum_{x\in\mathbb{Z}}v(x)\vec{\eta}^L(x)\frac{L-1}{L}\big\{\bar{\eta}(x+1)-\bar{\eta}(x+2)\big\}f(\eta)\nu_\rho(d\eta)\notag\\
&+\cdots+ 2\int \sum_{x\in\mathbb{Z}}v(x)\vec{\eta}^L(x)\frac{1}{L}\big\{\bar{\eta}(x+L-1)-\bar{\eta}(x+L)\big\}f(\eta)\nu_\rho(d\eta).\notag
\end{align}
In each one of the terms above, we write it as twice its half, and in one of the integrals we make the change $\eta$ to $\eta^{z,z+1}$ (for some suitable $z$), for which the measure $\nu_\rho$ is invariant. Thus, the last expression equals to
\begin{align}
&\int \sum_{x\in\mathbb{Z}}v(x)\vec{\eta}^L(x)\big\{\bar{\eta}(x)-\bar{\eta}(x+1)\big\}\big(f(\eta)-f(\eta^{x,x+1})\big)\nu_\rho(d\eta) \label{eq:lem6101}\\
+&\int \sum_{x\in\mathbb{Z}}v(x)\vec{\eta}^L(x)\frac{L-1}{L}\big\{\bar{\eta}(x+1)-\bar{\eta}(x+2)\big\}\big(f(\eta)-f(\eta^{x+1,x+2})\big)\nu_\rho(d\eta) \notag\\
+&\cdots+ \int \sum_{x\in\mathbb{Z}}v(x)\vec{\eta}^L(x)\frac{1}{L}\big\{\bar{\eta}(x+L-1)-\bar{\eta}(x+L)\big\}\big(f(\eta)-f(\eta^{x+L-1,x+L})\big)\nu_\rho(d\eta)\notag\\
+&\int\sum_{x\in\mathbb{Z}}v(x)\frac{\bar{\eta}(x+1)-\bar{\eta}(x)}{L}\big\{\bar{\eta}(x)-\bar{\eta}(x+1)\big\}f(\eta)\nu_\rho(d\eta). \label{eq:lem6102}
\end{align}
Notice that the last term  \eqref{eq:lem6102} comes from the change of variables $\eta$ to $\eta^{x,x+1}$ in the first term \eqref{eq:lem6101} above.
The whole sum can be rewritten as
\begin{align}
&\int \sum_{x\in\mathbb{Z}}v(x)\vec{\eta}^L(x)\frac{1}{L}\sum_{y=x+1}^{x+L}\sum_{z=x}^{y-1}\big\{\bar{\eta}(z)-\bar{\eta}(z+1)\big\}\big\{f(\eta)-f(\eta^{z,z+1})\big\}\nu_\rho(d\eta)\label{eq:boundlem}\\
-&\int\sum_{x\in\mathbb{Z}}v(x)\frac{1}{L}(\bar{\eta}(x)-\bar{\eta}(x+1))^2f(\eta)\nu_\rho(d\eta).\label{eq:secondpart}
\end{align}
Notice that the integral that we want to estimate in \eqref{eq:multline} is exactly equal to the sum of \eqref{eq:firstpart} and \eqref{eq:secondpart}, therefore it is bounded by the first term in the previous expression, namely \eqref{eq:boundlem}.
Now, we use the same arguments as above, namely, Young's inequality with $2A_x=L v(x)/n^a $ and we bound it by
\begin{equation*}
C(\rho)\frac{L}{ n^a} \sum_{x\in\mathbb{Z}}v^2(x)+\frac{n^a}{{L^2}} \sum_{x\in\mathbb{Z}}\sum_{y=x+1}^{x+L}\sum_{z=x}^{y-1}I^n_{z,z+1}(f).
\end{equation*}
From Lemma \ref{lemma:dir1} the proof ends.
\end{proof}

\begin{lemma}\label{lemma:dir1}

\quad

 For any $\ell\in\mathbb{N}$ it holds that
\begin{equation*}
\frac{1}{\ell^2}\sum_{x\in\mathbb{Z}}\sum_{y=x+2}^{x+\ell+1}\sum_{z=x+1}^{y-1}I^n_{z,z+1}(f)= D_n(f), \quad \text{and}\quad
\frac{1}{\ell}\sum_{x\in\mathbb{Z}}\sum_{z=y+x+1}^{y+x+\ell}I^n_{z,z+1}(f)= D_n(f).
\end{equation*}
\end{lemma}
\begin{proof}
The result follows from the translation invariance of the measure $\nu_\rho$, namely, Assumption \ref{ass1} and the fact that $\zeta^n_{z,z+1}(\eta)=\tau_z\zeta^n_{0,1}(\eta).$
\end{proof}

\subsection{Proof of Theorem \ref{theo:BG}}

Let $\ell_0 \leq L$. The idea of the proof consists in using the following decomposition of the local function
\begin{align}\bar{\eta}(x)\bar{\eta}(x+1)& -\Big(\big(\vec{\eta}^L(x)\big)^2-\frac{\chi(\rho)}{L}\Big)\notag\\
&=\bar{\eta}(x)\big(\bar{\eta}(x+1)-\vec{\eta}^{\ell_0}(x)\big) \label{eq:term1} \vphantom{\Big(}\\
& \quad + \vec\eta^{\ell_0}(x) \big(\bar\eta(x)-\vecleft\eta^{\ell_0}(x)\big)\label{eq:term2}  \vphantom{\Big(} \\
& \quad + \vecleft\eta^{\ell_0}(x) \big( \vec\eta^{\ell_0}(x)-\vec\eta^L(x)\big) \label{eq:term3}\vphantom{\Big(}\\
& \quad +\vec\eta^L(x) \big(\vecleft\eta^{\ell_0}(x)-\bar \eta(x)\big)\label{eq:term4}\vphantom{\Big(}\\
&\quad +\vec{\eta}^L(x)\bar{\eta}(x)
-\big(\vec{\eta}^L(x)\big)^2+ \frac{\big(\bar\eta(x)-\bar\eta(x+1)\big)^2}{2L}\label{eq:term5}\vphantom{\Big(}\\
&\quad -\frac{\big(\bar\eta(x)-\bar\eta(x+1)\big)^2}{2L}+\frac{\chi(\rho)}{L}.\label{eq:term6}\vphantom{\Big(}
\end{align}
The decomposition above involves six main terms, which we treat separately. The third term \eqref{eq:term3} is the most trickiest one, for which we need to perform a \emph{multi-scale analysis}. The fifth term \eqref{eq:term5} has already been estimated in Proposition \ref{prop:decomp}.

First, we start with the estimate of  \eqref{eq:term1},   \eqref{eq:term2}  and  \eqref{eq:term4}, for which we can use directly the one-block estimate: from Proposition \ref{prop:one-block}, applied successively with $\varphi(\tau_x\eta)=\bar{\eta}(x)$, $\psi(\tau_x\eta)=\vec \eta^{\ell_0}(x)$ and $\psi(\tau_x\eta)=\vec \eta^{L}(x)$ we get that
\begin{align*}
\mathbb{E}^n_{\rho}\Big[\Big(\int_{0}^t ds \sum_{x\in\mathbb{Z}}v(x)\bar{\eta}_{sn^a}(x)\big(\bar{\eta}_{sn^a}(x+1)-\vec{\eta}_{sn^a}^{\ell_0}(x)\big)\Big)^2\Big] & \leq
 C\frac{t\ell_0^2}{n^{a-1}}\|v\|_{2,n}^2,
\\
\mathbb{E}^n_{\rho}\Big[\Big(\int_{0}^t ds \sum_{x\in\mathbb{Z}}v(x)\vec\eta^{\ell_0}_{sn^a}(x)\big(\bar{\eta}_{sn^a}(x)-\vecleft{\eta}_{sn^a}^{\ell_0}(x)\big)\Big)^2\Big] & \leq
 C\frac{t\ell_0}{n^{a-1}}\|v\|_{2,n}^2,\\
\mathbb{E}^n_{\rho}\Big[\Big(\int_{0}^t ds \sum_{x\in\mathbb{Z}}v(x)\vec\eta^{L}_{sn^a}(x)\big(\bar{\eta}_{sn^a}(x)-\vecleft{\eta}_{sn^a}^{\ell_0}(x)\big)\Big)^2\Big] & \leq
 C\frac{t\ell_0^2}{n^{a-1}L}\|v\|_{2,n}^2.
\end{align*}

As we mentioned above, the hardest term to estimate is  \eqref{eq:term3} for which we need to do a  multi-scale analysis. For this term we have:
\begin{proposition}
\label{prop:final-size}
There exists a constant $C(\rho)$, such that, for any $\ell_0\leq L$, any $t>0$ and any function $v\in\ell^2(\mathbb{Z})$:
\begin{multline*}
\mathbb{E}^n_{\rho}\Big[\Big(\int_{0}^t ds \sum_{x\in\mathbb{Z}}v(x)\vecleft{\eta}^{\ell_0}_{sn^a}(x)\big(\vec{\eta}_{sn^a}^{\ell_0}(x+1)-\vec{\eta}_{sn^a}^{L}(x+1)\big)\Big)^2\Big]
\\ \leq
 C(\rho)\frac{t}{n^{a-1}}\Big\{L+\frac{\ell_0^2}{L}\Big\} \|v\|^2_{2,n}.
\end{multline*}
\end{proposition}
\begin{proof}
To prove the proposition, instead of replacing $\vec\eta^{\ell_0}(x)$ by $\vec\eta^L(x)$ in one step, we do it gradually, by doubling the size of the box of size $\ell_0$ at each step. For that purpose, let $\ell_{k+1}=2\ell_k$ and assume first that $L=2^M\ell_0$ for some $M\in\bb N$.
Then, rewrite \eqref{eq:term3} as 
\begin{align}
\vecleft\eta^{\ell_0}(x) \big(\vec\eta^{\ell_0}(x)-\vec\eta^L(x)\big) 
& = \sum_{k=0}^{M-1} \vecleft\eta^{\ell_k}(x) \big(\vec\eta^{\ell_k}(x)-\vec\eta^{\ell_{k+1}}(x)\big) \label{eq:dec1}\\
&  + \sum_{k=0}^{M-2} \vec\eta^{\ell_{k+1}}(x) \big(\vecleft\eta^{\ell_k}(x)-\vecleft\eta^{\ell_{k+1}}(x)\big) \label{eq:dec2}\\
& + \vec\eta^L(x) \big(\vecleft\eta^{\ell_{M-1}}(x)-\vecleft\eta^{\ell_0}(x)\big). \vphantom{\sum_{k}^M} \label{eq:dec3}
\end{align}
By a standard convexity inequality and using Minkowski's inequality twice, the expectation in the statement of the proposition is bounded from above by
\begin{align*}
&3\bigg\{\sum_{k=0}^{M-1}\bigg(\mathbb{E}^n_{\rho}\Big[\Big(\int_{0}^t ds \sum_{x\in\mathbb{Z}}v(x)\vecleft{\eta}^{\ell_k}_{sn^a}(x)\Big\{\vec{\eta}_{sn^a}^{\ell_k}(x)-\vec{\eta}_{sn^a}^{\ell_{k+1}}(x)\Big\}\Big)^2\Big]\bigg)^{1/2}\bigg\}^2\\
& + 3\bigg\{\sum_{k=0}^{M-2}\bigg(\mathbb{E}^n_{\rho}\Big[\Big(\int_{0}^t ds \sum_{x\in\mathbb{Z}}v(x)\vec{\eta}^{\ell_{k+1}}_{sn^a}(x)\Big\{\vecleft{\eta}_{sn^a}^{\ell_k}(x)-\vecleft{\eta}_{sn^a}^{\ell_{k+1}}(x)\Big\}\Big)^2\Big]\bigg)^{1/2}\bigg\}^2 \\
& +3\mathbb{E}^n_{\rho}\Big[\Big(\int_{0}^t ds \sum_{x\in\mathbb{Z}}v(x)\vec{\eta}^L_{sn^a}(x)\big(\vecleft{\eta}^{\ell_{M-1}}_{sn^a}(x)-\vecleft{\eta}_{sn^a}^{\ell_0}(x)\big)\Big)^2\Big].
\end{align*}
From Proposition \ref{doub box} with $\varphi(\tau_x\eta)=\vecleft{\eta}^{\ell_k}(x)$ and $\psi(\tau_x\eta)=\vec{\eta}^{\ell_{k+1}}(x)$, together with Corollary  \ref{prop:one-block3} with $\varphi(\tau_x\eta)=\vec\eta^L(x)$, the previous expression is bounded from above by
\[
 C(\rho)\frac{t}{n^{a-1}}\bigg\{\frac{\ell_0^2}{L}+2\sum_{k=0}^{M-1}\ell_k    \bigg\}\|v\|^2_{2,n}\leq  C(\rho)\frac{t}{n^{a-1}}\Big\{L+\frac{\ell_0^2}{L}\Big\}  \|v\|^2_{2,n}.
\]
 In the other cases we choose $M$ sufficiently big such that $2^M\ell_0\leq L\leq 2^{M+1}\ell_0$ and a similar computation to the one above proves the claim.
\end{proof}

The last term is easily estimated by using the Cauchy-Schwarz inequality.

\begin{proposition}\label{prop:cs2}
There exists a constant $C(\rho)>0$, such that for any $L\in\bb {N}$, any $t>0$ and any function $v\in\ell^2(\mathbb{Z})$:
\begin{equation*}
\mathbb{E}^n_{\rho}\Big[\Big(\int_{0}^t ds \sum_{x\in\mathbb{Z}}v(x)
\Big\{\frac{(\bar{\eta}_{sn^a}(x)-\bar{\eta}_{sn^a}(x+1))^2}{2L}-\frac{\chi(\rho)}{L}\Big\}\Big)^2\Big]\leq
 C(\rho)\frac{ t^2n}{L^2}\|v\|_{2,n}^2.
\end{equation*}
\end{proposition}

\section{Proof of the second-order Boltzmann-Gibbs principle for higher degree polynomial functions}\label{higher polynomial}
In this section we show how to extend Theorem \ref{theo:BG} to higher degree polynomials.
We consider the case of polynomial functions of degree three but the same result is true  for any higher degree. More precisely:
\begin{theorem}[Second-order Boltzmann-Gibbs principle for degree three polynomial functions]
\label{theo:BG3}
There exists a constant $C=C(\rho)>0$, such that for any $ L\in\N$, any  $t>0$, and any function $v\in\ell^2(\mathbb{Z})$:
\begin{multline}
\label{eq:BGexpo2}
\mathbb{E}^n_{{\rho}}\Big[\Big(\int_{0}^t \sum_{x\in\mathbb{Z}}v(x)\Big\{\bar{\eta}_{sn^{a}}(x-1)\bar{\eta}_{sn^{a}}(x)\bar{\eta}_{sn^{a}}(x+1)-\ (\vec{\eta}_{sn^a}^L(x))^3+\frac{\xi(\rho)}{L^2}\Big\} ds\Big)^2\Big]\\
\leq Ct\Big\{\frac{L}{n^{a-1}} +\frac{tn}{L^2}\Big\}\|v\|_{2,n}^2
\end{multline}
where \[
\xi(\rho):=\int_\Omega\big(\eta(x)-\rho\big)^3\;\nu_\rho(d\eta).\]
\end{theorem}

The idea of the proof is similar to the one used above: it consists in using the following decomposition of the local function
\begin{align}
&\bar{\eta}(x-1)\bar{\eta}(x)\bar{\eta}(x+1) -(\vec{\eta}^L(x))^3+\frac{\xi(\rho)}{L^2}\label{eq:1}\\
&=\bar{\eta}(x-1)\Big(\bar{\eta}(x)\bar{\eta}(x+1)-\big(\vec{\eta}^L(x)\big)^2+\frac{\chi(\rho)}{L}\Big)\label{eq:2}\\
& \quad +\big(\vec{\eta}^L(x)\big)^2 \Big(\bar{\eta}(x-1)-\vec{\eta}^L(x-1)\Big)\label{eq:3}\\
&\quad +\frac{L-1}{L}\Big(\frac{(\eta(x+1)-\eta(x))^3}{2L^2}+\vec{\eta}^L(x)\frac{(\eta(x)-\eta(x+1))^2}{L}\Big)\label{eq:3bis}\\
&\quad-\frac{L-1}{L}\Big(\frac{(\eta(x+1)-\eta(x))^3}{2L^2}\Big)\label{eq:4}\\
& \quad - \frac{L-1}{L}\Big(\vec{\eta}^L(x)\frac{(\eta(x)-\eta(x+1))^2}{L}-\frac{\xi(\rho)}{L^2}\Big)\label{eq:6}\\
& \quad - \bar{\eta}(x-1) \frac{\chi(\rho)}{L} - \frac{\xi(\rho)}{L^2}\Big(\frac{L-1}{L}-1\Big)\label{eq:6bis}\\
&\quad + {\big(\vec\eta^L(x)\big)^2}\Big(\vec\eta^L(x-1)-{\vec\eta^L(x)}\Big).\label{eq:7}
\end{align}
The first term \eqref{eq:2} can be  treated as in Theorem \ref{theo:BG} and it gives an error of order $Ct\big\{L/n^{a-1}+tn/L^2\big\}$. In order to help the reader to check that claim, let us notice that \eqref{eq:2} rewrites as
\begin{align*}
& \bar{\eta}(x-1)\Big(\bar{\eta}(x)\bar{\eta}(x+1)-\big(\vec{\eta}^L(x)\big)^2+\frac{\chi(\rho)}{L}\Big) \\
& = \Big(\bar\eta(x-1)-\vecleft\eta^{\ell_0}(x-1)\Big)\Big(\bar{\eta}(x)\bar{\eta}(x+1)-\big(\vec{\eta}^L(x)\big)^2+\frac{\chi(\rho)}{L}\Big)\\
& \quad + \vecleft\eta^{\ell_0}(x-1)\bar\eta(x)\Big(\bar\eta(x+1)-\vec\eta^{\ell_0}(x)\Big)\\
&\quad +  \vecleft\eta^{\ell_0}(x-1) \bar\eta(x)\Big(\vec\eta^{\ell_0}(x)-\vec\eta^L(x)\Big)\\
& \quad + \vecleft\eta^{\ell_0}(x-1) \Big(\bar\eta(x)\vec\eta^L(x) - \big(\vec\eta^L(x)\big)^2+\frac{\chi(\rho)}{L}\Big).
\end{align*}
Now notice that
\[
\int_\Omega(\eta(x)-\eta(x+1))^3\nu_\rho(d\eta)=0,
\]
and
\[
\int_\Omega \vec{\eta}^L(x)\frac{(\eta(x)-\eta(x+1))^2}{L}\nu_\rho(d\eta)=\frac{\xi(\rho)}{L^2}.
\]
Therefore the terms \eqref{eq:4}, \eqref{eq:6} and \eqref{eq:6bis} are treated with the Cauchy-Schwarz inequality as in Proposition \ref{prop:cs2}. In the same way, the term \eqref{eq:7}  can be easily treated with the Cauchy-Schwarz inequality and using independence with respect to the invariant measure $\nu_\rho$. All of them  give an error of order at most $t^2n/L^2$. The only term that requires a little bit of work is \eqref{eq:3}+\eqref{eq:3bis}. This is the content of the following proposition.

\begin{proposition}
There exists a constant $C=C(\rho)>0$, such that for any $ L\in\N$, any  $t>0$, and any function $v\in\ell^2(\mathbb{Z})$:
\begin{equation*}
\begin{split}
\mathbb{E}^n_{\rho}\Big[\Big(\int_{0}^t ds \sum_{x\in\mathbb{Z}}&v(x)\Big[\big(\vec{\eta}^L(x)\big)^2\big(\bar{\eta}(x-1)-\vec{\eta}^L(x-1)\big)\\
&+\frac{L-1}{L}\Big\{\frac{(\eta(x+1)-\eta(x))^3}{2L^2}+\vec{\eta}^L(x)\frac{(\eta(x)-\eta(x+1))^2}{L}\Big\}\Big]\Big)^2\Big] \\&\leq
 C(\rho)\frac{tL}{n^{a-1}}\|v\|_{2,n}^2.
\end{split}
\end{equation*}
\end{proposition}

\begin{proof}
By \cite[Lemma 2.4]{KLO}, we can bound the previous expectation from above by
\begin{multline*}
Ct\;\sup_{f\in {\bb L}^2(\nu_\rho)}\Big\{2\int \sum_{x\in\mathbb{Z}}v(x)\Big[\big(\vec{\eta}^L(x)\big)^2 \big(\bar{\eta}(x-1)-\vec{\eta}^L(x-1)\big)
+  \frac{L-1}{L} \times\\\Big\{\frac{(\eta(x+1)-\eta(x))^3}{2L^2}+\vec{\eta}^L(x)\frac{(\eta(x)-\eta(x+1))^2}{L}\Big\}\Big]f(\eta)\nu_\rho(d\eta)-n^aD_n(f)\Big\}.
\end{multline*}
As above, we notice that
\begin{align*}
&2\int \sum_{x\in\mathbb{Z}}v(x)\big(\vec{\eta}^L(x)\big)^2\big\{\bar{\eta}(x-1)-\vec{\eta}^L(x-1)\big\}f(\eta)\nu_\rho(d\eta)\\
& =2\int \sum_{x\in\mathbb{Z}}v(x)\big(\vec{\eta}^L(x)\big)^2\big\{\bar{\eta}(x-1)-\bar{\eta}(x)\big\}f(\eta)\nu_\rho(d\eta)\\
&+2\int \sum_{x\in\mathbb{Z}}v(x)\big(\vec{\eta}^L(x)\big)^2\;\frac{L-1}{L}\big\{\bar{\eta}(x)-\bar{\eta}(x+1)\big\}f(\eta)\nu_\rho(d\eta)\\
&+\cdots+ 2\int \sum_{x\in\mathbb{Z}}v(x)\big(\vec{\eta}^L(x)\big)^2\;\frac{1}{L}\big\{\bar{\eta}(x+L-2)-\bar{\eta}(x+L-1)\big\}f(\eta)\nu_\rho(d\eta).
\end{align*}
By writing each term as twice its half and doing the  exchange $\eta$ to $\eta^{z,z+1}$ (for some suitable $z$), last expression becomes equal to
\begin{align}
&\int \sum_{x\in\mathbb{Z}}v(x)\big(\vec{\eta}^L(x)\big)^2\big\{\bar{\eta}(x-1)-\bar{\eta}(x)\big\}\big(f(\eta)-f(\eta^{x-1,x})\big)\nu_\rho(d\eta) \notag\\
+&\int \sum_{x\in\mathbb{Z}}v(x)\big(\vec{\eta}^L(x)\big)^2\;\frac{L-1}{L}\big\{\bar{\eta}(x)-\bar{\eta}(x+1)\big\}\big(f(\eta)-f(\eta^{x,x+1})\big)\nu_\rho(d\eta) \label{eq:lem6101bisbis}\\
+&\cdots +\notag\\+& \int \sum_{x\in\mathbb{Z}}v(x)\big(\vec{\eta}^L(x)\big)^2\;\frac{\bar{\eta}(x+L-2)-\bar{\eta}(x+L-1)}{L} \big(f(\eta)-f(\eta^{x+L-2,x+L-1})\big)\nu_\rho(d\eta)\notag\\
-&\int\sum_{x\in\mathbb{Z}}v(x)\frac{L-1}{L}\Big\{\frac{(\bar{\eta}(x+1)-\bar{\eta}(x))^3}{L^2}+2\vec{\eta}^L(x)\frac{\big(\bar{\eta}(x)-\bar{\eta}(x+1)\big)^2}{L}\Big\}f(\eta)\nu_\rho(d\eta).\notag
\end{align}
The last term comes from the change of variables $\eta$ to $\eta^{x,x+1}$ in the second term \eqref{eq:lem6101bisbis}.
The whole expression above can be rewritten as
 \begin{align}
&\int \sum_{x\in\mathbb{Z}}v(x)\big(\vec{\eta}^L(x)\big)^2\frac{1}{L}\sum_{y=x}^{x+L-1}\sum_{z=x-1}^{y-1}\big\{\bar{\eta}(z)-\bar{\eta}(z+1)\big\}\big(f(\eta)-f(\eta^{z,z+1})\big)\nu_\rho(d\eta) \label{eq:firstbis}\\
-&\int\sum_{x\in\mathbb{Z}}v(x)\frac{L-1}{L}\Big\{\frac{(\bar{\eta}(x+1)-\bar{\eta}(x))^3}{L^2}+2\vec{\eta}^L(x)\frac{\big(\bar{\eta}(x)-\bar{\eta}(x+1)\big)^2}{L}\Big\}f(\eta)\nu_\rho(d\eta).\notag
\end{align}
Now, the integral that we want to control is bounded by \eqref{eq:firstbis}.
By the same arguments as above, namely, by Young's inequality with $2A_x= L v(x)/n^a $, we bound it by
\begin{equation*}
C(\rho)\frac{L}{ n^a} \sum_{x\in\mathbb{Z}}v^2(x)+\frac{n^a}{{L^2}} \sum_{x\in\mathbb{Z}}\sum_{y=x+1}^{x+L}\sum_{z=x}^{y-1}I^n_{z,z+1}(f).
\end{equation*}
From Lemma \ref{lemma:dir1} the proof ends.
\end{proof}

\begin{acknowledgement}
This work benefited from the support of the project EDNHS
ANR-14-CE25-0011 of the French National Research Agency (ANR).  PG thanks FCT/Portugal for support through the project UID/MAT/04459/2013. The work of MS was supported by  CAPES (Brazil) and IMPA (Instituto de Matematica Pura e Aplicada, Rio de Janeiro) through a post-doctoral fellowship, and  in part by the Labex CEMPI (ANR-11-LABX-0007-01).
\end{acknowledgement}

\end{document}